\documentclass[11pt]{article}
\usepackage[colorlinks,linkcolor=blue,citecolor=blue]{hyperref}
\usepackage[abs]{overpic}
\usepackage[normalem]{ulem} 
\usepackage{amsmath, amscd, amssymb, amsthm,amsfonts}
\usepackage{euscript, mathrsfs, latexsym,mathabx,stmaryrd}
\usepackage{color}
\usepackage{hyperref} 

\usepackage{lmodern}
\usepackage[T1]{fontenc} 

\usepackage{multicol}

\ifx\pdfoutput\undefined
\usepackage{graphicx}
\else
\fi

\ifx\xetex
\usepackage{fontspec}
\usefont{EU1}{lmr}{m}{n}
\else
\fi

\usepackage{url} 
\usepackage{tikz}
\usetikzlibrary{matrix,arrows,positioning}
\tikzset{node distance=2cm, auto}

\usepackage[nodayofweek]{datetime}

\parskip=\medskipamount

\newcommand{\conj}[1]{\quad\textnormal{ #1 }\quad}

\newcommand{\inp}[1]{\ensuremath{\langle #1 \rangle}}

\newcommand{\normaltext}[1]{\textnormal{#1}}

\makeatletter
\def\imod#1{\allowbreak\mkern2.5mu({\operator@font mod}\,#1)}
\makeatother

\renewcommand{\a}{\alpha}
\renewcommand{\b}{\beta}
\newcommand{\e}{\epsilon}

\newcommand{\opp}{\oplus}
\newcommand{\ott}{\otimes}

\newcommand{\ga}{\gamma}
\renewcommand{\th}{\theta}

\newcommand{\CPPic}[1]{
\begin{minipage}{.8in}
\includegraphics[scale=1.1]{#1}
\end{minipage}
}

\newcommand{\aF}{\mathcal{F}}

\newcommand{\aK}{\mathcal{K}}
\newcommand{\aL}{\mathcal{L}}

\usepackage{bbm}

\newcommand{\CC}{\mathbb{C}}

\newcommand{\HH}{\mathbb{H}}

\newcommand{\RR}{\mathbb{R}}

\newcommand{\ZZ}{\mathbb{Z}}

\theoremstyle{plain}

\newtheorem{thm}{Theorem}[section]

\newtheorem{theorem}[thm]{Theorem}

\newtheorem{question}[thm]{Question}

\newtheorem{prop}[thm]{Proposition}

\newtheorem{cor}[thm]{Corollary}
\newtheorem{lemma}[thm]{Lemma}

\theoremstyle{remark}

\theoremstyle{definition}

\newtheorem{defn}[thm]{Definition}

\newtheorem{definition}[thm]{Definition}

\newtheorem{rmk}[thm]{Remark}
\newtheorem{remark}[thm]{Remark}

\numberwithin{equation}{section}

\makeatletter
\def\imod#1{\allowbreak\mkern2.5mu({\operator@font mod}\,#1)}
\makeatother

\makeatletter
\def\namedlabel#1#2{\begingroup
   \def\@currentlabel{#2}%
   \label{#1}\endgroup
}
\makeatother

\newcommand{\cL}{\mathscr{L}\!}

\newcommand{\vnp}[1]{\lvert #1 \rvert}

\newcommand{\I}{[0,1]}
\newcommand{\xto}[1]{\xrightarrow{#1}}

\usepackage{extarrows}

\theoremstyle{plain}

\newcommand{\im}{im}

\newcommand{\go}[1]{\mathfrak{g}(#1)}
\newcommand{\gor}[1]{\mathfrak{g}_R(#1)}
\newcommand{\gs}{\go{S}}
\newcommand{\grs}{\gor{S}}
\newcommand{\hqs}{H_q(S)}

\newcommand{\F}{\aF}

\newcommand{\g}{\mathfrak{g}}

\newcommand{\freehtpy}{\hat{\pi}}
\newcommand{\reghtpy}{\hat{\pi}_R}

\usepackage{caption}

\newtheorem*{theorem*}{Theorem}

\begin{document}

\title{Gradation of Algebras of Curves by the Winding Number}
\author{Mohamed Imad Bakhira and Benjamin Cooper}

\date{}
\maketitle

\newcommand{\Addresses}{{
  \bigskip
  \footnotesize
\textsc{University of Iowa, Department of Mathematics, 14 MacLean Hall, Iowa City, IA 52242}\newline
  \textit{E-mail:} \texttt{mohamedimad-bakhira\char 64uiowa.edu} \newline
\textit{E-mail:} \texttt{ben-cooper\char 64 uiowa.edu}
}}

\begin{abstract}
  We construct a new grading on the Goldman Lie algebra of a closed oriented
  surface by the winding number. This grading induces a grading on the
  HOMFLY-PT skein algebra and related algebras. Our work supports the conjectures
  of B. Cooper and P. Samuelson \cite{cooper}.
\end{abstract}

\section{Introduction}\label{introsec}

The Goldman Lie algebra $\gs$ of a surface $S$ is the Lie algebra of free
homotopy classes of loops. The product is given by summing the signed
concatenations of curves over their points of intersection. This Lie algebra
encodes a wealth of information about intersection numbers of curves
\cite{chas}, maintains a close relation to skein invariants \cite{turaev}
and invariant functions on spaces of surface group representations
\cite{goldman}.  It is also a fundamental component of string topology
\cite{chas_sullivan}.  In this article, we construct a new cyclic grading on
the Goldman Lie algebras and the HOMFLY-PT skein algebras of surfaces. In
particular, there is a product preserving decomposition
$$\gs = \bigoplus_{a\in\ZZ/\chi} \gs_a$$
where $\chi = \chi(S)$ is the Euler characteristic of $S$. 
The existence of a cyclic grading by the
winding number is predicted by recent work of B. Cooper and P. Samuelson
\cite{cooper}. This work relates the HOMFLY-PT skein algebra to the Hall algebra
of the Fukaya category.  For closed surfaces, these conjectures cannot be
verified directly because the lack of $\mathbb{Z}$-grading on the Fukaya
category currently impedes a rigorous study of their Hall algebras. In this way,
our work constitutes new evidence for these conjectures. Presently, the only
non-trivial evidence for closed surfaces $S$ of genus greater than one. This new
grading may also allow us to glean new information about these algebras
and their many connections to other areas of mathematics.

In the remainder of the introduction, we explain what is meant by
winding number, we discuss the conjectural context for our construction
and we present an outline of our approach to the construction
of the grading.

\subsection{Winding numbers}\label{windingsec}
The winding number of an closed oriented immersed curve in a surface is the total signed number of revolutions that its tangent vector undergoes in one traversal. In the plane this is a well-defined integer, but the generalization to closed surfaces of genus $g > 1$ has some indeterminacy. For a survey of winding number see \cite{MC}. A brief introduction is provided below.

In his study of regular closed curves in the plane, H. Whitney showed that the planar winding number is invariant under regular homotopy \cite{whitney}. 
In his thesis work \cite{smale}, S. Smale showed that for a Riemannian manifold $M$, 
$$Imm_p(S^1,M) \simeq \Omega_{p'}S(TM),$$
the space $Imm_p(S^1,M)$ of immersed loops at the basepoint $p=(q,v)\in TM$ with initial and terminal velocity $v\in T_q M$ is weakly homotopic to the space of loops in the unit tangent bundle $S(TM)$ with basepoint 
$p'=(q,v/\vnp{v})$.  This equivalence determines isomorphisms
\begin{equation}\label{smaleeqn}
\pi_0 (Imm_p(S^1,M))\cong \pi_0 (\Omega_{p'}S(TM)) \cong \pi_1 (S(TM),p').
\end{equation}
The {\em regular fundamental group} $\pi_R(M,p)$ is defined to be this group
$$\pi_R (M,p):=\pi_0 (Imm_p(S^1,M)).$$
In words, $\pi_R(M,p)$ is the group of regular homotopy classes of curves based at $p$ in $M$ under the operation of loop concatentation, see Def. \ref{loopproddef}. Eqn. \eqref{smaleeqn} implies that $\pi_R(M,p)$ is independent of basepoint. 

H. Seifert computed the group $\pi_1(S(TS))$ for a closed orientable surface $S$ of genus $g$ \cite{seifert}. Recall that such a surface $S$ can be expressed as a quotient of the $4g$-gon in which the $i$th edge is identified with the $i+2$nd edge for each $0\leq i < 4g$ such that $i\equiv 0,1\imod{4}$. This decomposition gives us the presentation
\begin{equation}\label{pipres}
  \pi_1 (S) \cong \inp{a_1,b_1,\ldots,a_g,b_g \mid r } \conj{ where } r=\prod_{i=1}^{g}[a_i,b_i].
\end{equation}  
and $[a,b]= aba^{-1}b^{-1}$.
If $f = s^{-1}(q)$ denotes the homotopy class of a fiber of the unit circle bundle $s : S(TS)\to S$ then the fundamental group of the unit tangent bundle has a compatible presentation
\begin{equation}\label{spres}
  \pi_1(S(TS))\cong
  \inp{a_1,b_1,...,a_g,b_g,f \mid f^{2g-2} r\textnormal{ and } [f,a_i], [f,b_i] \textnormal{ for } 1\leq i \leq  g}.
  \end{equation}

From this perspective, the {\em winding number} $\omega(\ga)$ of a regular
closed curve $\ga\in \pi_R(S)$ is its projection onto the torsion component
$\ZZ/\chi \subset \pi_R(S)^{ab}$ of the abelianization. In more detail, the
abelianization can be computed from Equation \eqref{spres} using Smale's
isomorphism
\begin{equation}\label{stareq}
  \pi_R(S)^{ab}\cong\pi_1(S(TS))^{ab} \cong \ZZ^{2g}\oplus \ZZ/\chi \conj{ where } \chi = 2g-2
  \end{equation}
and the observation that $r=1$ in the abelianization. The torsion group is
generated by the fiber $f$. The winding number
homomorphism
\begin{equation}\label{windingnumbereq}
  \omega : \pi_R (S)\to \mathbb{Z}/\chi
  \end{equation}
first maps $\ga \in \pi_R(S)$ to the abelianization $\pi_R(S)^{ab}$ and then
extracts the coefficient of the fiber $f$.

There are many curves with non-zero winding number.  For instance, the separating curve on the genus 2 surface featured in Figure \ref{fig:nontrivialwinding} has winding number $1\imod{2}$.
\begin{figure}[h]
\begin{center}
\begin{overpic}[scale=0.8]
{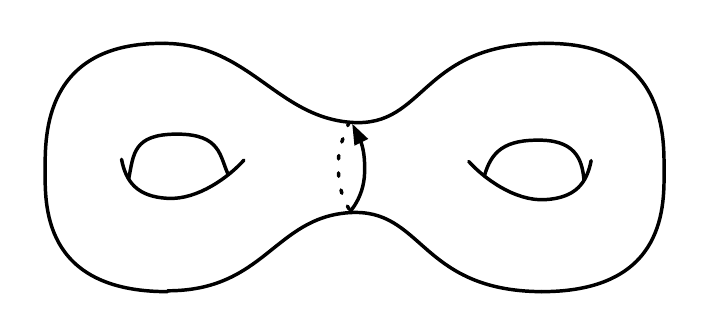}
\put(77,63){$\ga$}
\end{overpic}
\end{center}
  \caption{A curve $\ga$ with winding number $\omega(\ga) = 1 \in\ZZ/2$}
  \label{fig:nontrivialwinding}
\end{figure}

In his work on winding numbers of regular curves on surfaces B. L. Reinhart
introduced an integral presentation for the winding number
\begin{equation}\label{rheq}
  \omega(\ga)=\frac{1}{2\pi}\int_{S^1} A_\ga^* (d\theta) \quad\quad\imod {\chi}
\end{equation}  
as the degree of the map determined by the angle $A_\ga$ between the derivative
of $\ga$ and a vector field $S$, see Definition \ref{RHdef} and
\cite{reinhart}. Reinhart's integral satisfies the property that
$\omega(\gamma)=0$ for each of the regular representatives of generators
$\gamma \in\{a_i,b_i\}_{i=1}^g$ and $\omega(\eta)=1$ for a small
counterclockwise contractible regular loop $\eta$. This makes it compatible
with the definition above, as the generators $\{a_i,b_i\}_{i=1}^{g}$
determine the $\ZZ^{2g}$-component of $\pi_R(S)^{ab}$, while the fiber $f$
corresponds to a small homotopically trivial counterclockwise loop. See Section \ref{rhsec} for further discussion.

\begin{rmk}
The winding number depends on the splitting isomorphism in Eqn. \eqref{stareq}. Although the arguments in this paper hold for any such isomorphism, we will use the one determined by the choice of $\{a_i,b_i\}$ basis mentioned in the above. For some applications, it may be natural to grade algebras by $H_1(S(TS))$ rather than $\ZZ^{2g} \opp \ZZ/\chi$.
  \end{rmk}

\subsection{Context}\label{contextsec}
The work of B. Cooper and P. Samuelson establishes a relationship between
the HOMFLY-PT skein algebra of a surface and the Hall algebra of the Fukaya
category of the surface \cite{cooper}. For certain surfaces with boundary, they
prove that elements in the Hall algebra of the Fukaya category satisfy the
HOMFLY-PT skein relation. It is natural to ask whether there is evidence for
a relationship between the Hall algebra of the Fukaya category of a {\em
  closed surface} $S$ and the skein algebra. The theorem of H. Morton and
P. Samuelson which relates the skein algebra of the torus and the elliptic
Hall algebra \cite{peter} provides some evidence in genus one when one
assumes that a version of homological mirror symmetry holds over finite
fields. Our construction of a grading by winding number on the HOMFLY-PT
skein algebra provides evidence for this conjecture when the surfaces are
closed and $\chi < 0$.

Since the Hall algebra of a category $\mathcal{C}$ is always graded by the
Grothendieck group $K_0(\mathcal{C})$ of the category, the Hall algebra of the Fukaya
category $\F(S)$ should be graded by the group $K_0(\F(S))$. M. Abouzaid
computed this group for closed surfaces
$$K_0(\F(S))\cong H_1(S(TS))\oplus\mathbb{R} \cong \mathbb{Z}^{2g}\oplus \mathbb{Z}/\chi \oplus\mathbb{R}$$
when the Euler characteristic satisfies $\chi<0$ \cite{abouzaid}. So if the Hall algebra
of the Fukaya category of a closed surface could be defined
then it would be graded by the group $H_1(S) \opp \mathbb{Z}/\chi$. 
Because of these observations, it was conjectured that the HOMFLY-PT skein algebra shares this grading.

Now, it is considered well-known that one can grade skein algebras by the
first homology group $H_1(S)$ of the surface. However, a grading by the
winding number $\omega$ is somewhat counterintuitive because the curves
which constitute elements of the skein algebra live in the 3-dimensional
space $S\times [0,1]$, so Reinhart's integral \eqref{rheq} is not
applicable. It is in this sense that the $\ZZ/\chi$-grading by winding
number on the HOMFLY-PT skein algebra extends and supports the conjectures
of B. Cooper and P. Samuelson.

\subsection{Results}\label{resultsec}
The principal result of this paper is to establish a grading of the Goldman Lie algebra $\gs$ and its quantization $\hqs$ the HOMFLY-PT skein algebra by the winding number. The theorem below summarizes our main results.
\begin{theorem*}
Let $S$ be a compact connected oriented surface. There is a canonical
extension of the winding number homomorphism
$$\omega : \pi_R (S)\to \mathbb{Z}/\chi\conj{ $\rightsquigarrow$} \hat{\pi}(S)\to\mathbb{Z}/\chi$$ 
to a map on the set $\hat{\pi}(S)$ of free homotopy classes of loops on $S$. 
This extension determines a grading on the Goldman Lie algebra
$$\gs = \bigoplus_{a\in \ZZ/\chi} \mathfrak{g}(S)_a \conj{ where } [\mathfrak{g}(S)_a,\mathfrak{g}(S)_b] \subset \mathfrak{g}(S)_{a+b}.$$
There are similar gradings on the regular Goldman Lie algebra $\grs$ and the HOMFLY-PT skein algebra $\hqs$.
\end{theorem*}

 The main obstacle to the construction of the grading is the domain
 of definition: the notion of a winding number is defined on regular
 homotopy classes whereas the Goldman Lie algebra is defined on free
 homotopy classes of curves. 
The winding number of a curve is not well-defined on curves up to free
homotopy because there are multiple regular homotopy classes within any
given free homotopy class.
More concretely, topological artifacts
 such as fish-tails can be introduced to artificially increase or decrease
 the winding number of a curve, see Section \ref{absec}.
\begin{figure}[h]
\begin{center}
  \includegraphics[scale=.75]{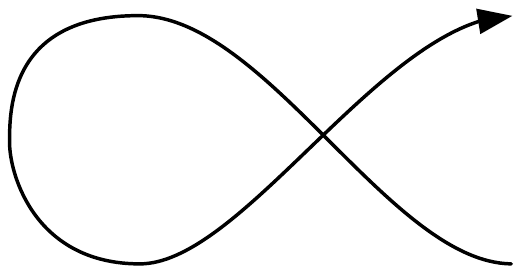}
  \caption{A fish-tail}
  \label{fig:fishtail}
\end{center}
\end{figure}

{\bf Organization} Section \ref{notationsec} reviews concepts which will be
important in the main construction. In Section \ref{constructionsec}, the
relationship between the set of free regular homotopy classes of curves
$\hat{\pi}_R(S)$ and free homotopy classes of curves $\hat{\pi}(S)$ is
articulated. In \S\ref{bundlemapssec}, we show that there is a surjective
loop product preserving map $s_* :\hat{\pi}_R(S) \to \hat{\pi}(S)$, the
fiber of which is a $\ZZ$-torsor generated by the operation of adding or
removing fish-tails. In \S\ref{bijectionsec}, we show that the map $s_*$
admits a section
$\Phi : \hat{\pi}(S) \to \hat{\pi}_{U,R}(S) \subset \hat{\pi}_R(S)$, taking
free homotopy classes of curves to unobstructed representatives.  In
\S\ref{windgoldsec}, the winding grading on the Goldman Lie algebra $\gs$ is
introduced by setting the winding number of a non-contractible free homotopy
class $[\a]$ to be the winding number of its unobstructed representative
$\Phi([\a])$.  We check that this definition respects concatenation of
loops, extends additively over the Lie bracket and so defines a map
$\omega : \gs \to \ZZ/\chi$ which determines a grading of the Goldman Lie
algebra. Section \ref{regoldmansec} introduces the regular Goldman Lie
algebra $\grs$. In Section \ref{homflyptsec}, we prove that the winding
number grading of the regular Goldman Lie algebra extends to the HOMFLY-PT
skein algebra.

{\bf Acknowledgements} The second author would like to thank P. Samuelson
for the friendly conversations and suggestions. Both authors would like to
thank C. Frohman and J. Greene, as well as the referee for the careful
reading.

\section{Notations and definitions}\label{notationsec}

We recall a few basic definitions, the definition of the Goldman Lie algebra, Reinhart's integral definition of the winding number and some results about unobstructed curves.

\subsection{Grading}\label{gradingsec}
If $M$ is an $R$-module and $A$ is an abelian group then a {\em grading} of $M$ by $A$ is a direct sum decomposition $M \cong \opp_{a\in A} M_a$.  If $M$ is an algebra then we require that $mn \in M_{a+b}$ when $m\in M_a$ and $n\in M_b$. Likewise, if $M$ is a Lie algebra then we require that $[m,n]\in M_{a+b}$ when $m\in M_a$ and $n\in M_b$.

If $S$ is a set and $M=R\inp{S}$ is the free $R$-module on $S$ then a grading of $M$ by $A$ is determined by a map $gr : S \to A$; in this case, $M_a := R\inp{m \mid gr(m) = a}$. If $M$ is a Lie algebra then $[m,n]\in M_{a+b}$ when $gr(m) = a$ and $gr(n) = b$.

\subsection{Surface topology}\label{surfacetopsec}
Throughout this paper, $S$ will always be a closed connected oriented
surface.  A set of curves $\{ \ga_i : S^1 \to S \}_{i\in I}$ on $S$ is said
to be in general position when all curves are normal closed, all
intersections are transverse and there are no triple intersections among
curves. We will always assume curves are in general position.

Two curves $\a,\b : S^1\to S$ are {\em freely homotopic} when they are in
the same path component of the free loop space $Map(S^1,S)$, they are {\em
  homotopic} when they share a basepoint $p$ and are contained in the same
path component of the based loop space $\Omega_pS=Map_p(S^1,S)$.  Two embeddings
$\a$ and $\b$ are {\em isotopic} when they are contained in the same path
component of the space of embeddings $Emb(S^1,S)$. Two immersions are {\em
  freely regularly homotopic} when they are contained in the same path
component of the space of immersions $Imm(S^1,S)$ and {\em regularly
  homotopic} when they are contained in the same path component of the space
of immersions $Imm_p(S^1,S)$ based at $p$.

The set of {\em free homotopy classes of loops on $S$} will be denoted by
$\hat{\pi}(S)$. A free homotopy class of map in $\hat{\pi}(S)$ can also be
thought of as a conjugacy class of $\pi_1(S)$.  If $\ga : S^1\to S$ is a
loop then we will denote by $[\ga]$ its corresponding free homotopy class.
Any other equivalence relation on curves will be indicated by a subscript on
the brackets, for example
$$[\ga]_{\pi_1}\in\pi_1(S,p)$$ 
means that $\ga : S^1\to S$ should be seen as an equivalence class $[\gamma]$ of curves in $S$ based at $p$ in $\pi_1(S,p)$.

\subsection{The Goldman Lie algebra}

Here we recall the definition of the Goldman Lie algebra. 
 The first definition recalls how we will concatenate curves.

\begin{defn}{($\a\cdot_p\b$)}\label{loopproddef} 
Suppose that two loops $\a,\b : S^1 \to S$ cross at a point $p\in \a\cap \b$ then $\a$ and $\b$ define elements $[\a]_{\pi_1}, [\b]_{\pi_1} \in \pi_1(S,p)$ and their {\em oriented loop product} $[\alpha \cdot_p\beta]_{\pi_1} := [\a]_{\pi_1}[\b]_{\pi_1}$ is just their product in $\pi_1(S,p)$. More generally, $[\a \cdot_p \b]$ will denote the image of $[\alpha \cdot_p\beta]_{\pi_1}$ in $\hat{\pi}(S)$.  The free homotopy class of $\a\cdot_p\b$ does not depend on the choice of representatives $\a\in[\a]$ or $\b\in[\b]$.
\end{defn}
Up to free homotopy, in a small neighborhood $U$ of $p\in\a\cap\b$, the product $\alpha\cdot_p\beta$ can be viewed as replacing the picture on the lefthand side of Figure \ref{fig:loopprod} with the righthand side.
\begin{figure}[h]
\begin{center}
\begin{overpic}[scale=0.6]
{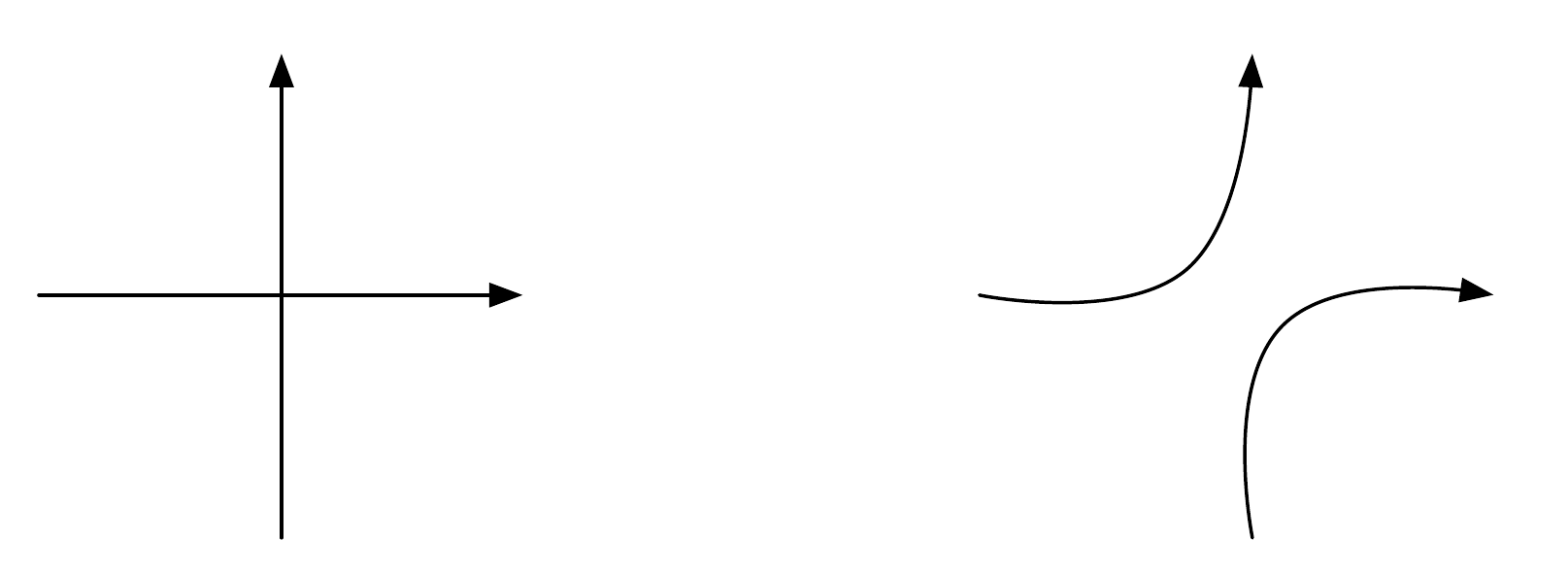}
\put(125,47){$\rightsquigarrow $}
\put(55,80){$\a$}
\put(8,38){$\b$}
\put(230,80){$\a\cdot_p\b$}
\put(55,38){$p$}
\end{overpic}
\end{center}
  \caption{The loop product}
  \label{fig:loopprod}
\end{figure}
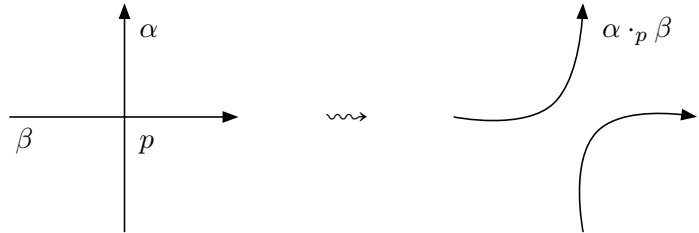

This oriented loop product preserves immersions because the two diagrams agree on vectors in the boundary $\partial U$.

\begin{defn}{($\e_p$)}
If two immersed curves $\a,\b : (0,1) \to S$ intersect at a point $p$ then the {\em sign} $\e_p(\a,\b)$ of their intersection is $+1$ if the ordered pair of derivatives $(\dot{\a}(p),\dot{\b}(p))$ agrees with the orientation of $S$ at $p$ and $\e_p(\a,\b) := -1$ otherwise.
\end{defn}

The loop product and the sign are important to us because they are used to define the Goldman Lie algebra of the surface.

\begin{defn}[Goldman Lie algebra]\label{goldmandef}
The {\em Goldman Lie algebra} $\gs$
of a surface $S$ is the free abelian group on the set $\hat{\pi}(S)$ of free homotopy classes of curves equipped with a Lie bracket
$$\gs := \ZZ\inp{\hat{\pi}(S)}.$$
Given two classes of curves $[\a],[\b]\in \hat{\pi}(S)$, we choose
representatives $\a\in[\a]$ and $\b\in[\b]$ so that the Goldman Lie bracket is defined to be the signed
sum of their oriented loop products at points $p$ of intersection
\begin{equation}\label{goldmanbracket}
  [[\alpha],[\beta]] := \sum _{p\in\alpha\cap\beta} \epsilon_p(\alpha,\beta) [\alpha \cdot_p \beta],
  \end{equation}
see \cite[Thm. 5.3]{goldman} for further details.
\end{defn}

\subsection{Reinhart's winding number}\label{rhsec}
As mentioned in the introduction, B. L. Reinhart used differential topology
to give a definition of the winding number $\omega : \pi_R(S)\to \ZZ/\chi$,
see \cite{reinhart}. Here we recall a few more details.

\begin{defn}[Winding number homomorphism]\label{RHdef}
  Choose a Riemannian metric on $S$ and a vector field $X$ with isolated
  zeros. Let $\angle_p(Y,Z)$ denote the angle between two vectors $Y$
  and $Z$ in $T_pS$ at $p\in S$.

 If $\ga(\theta): S^1\to S$ is an
  immersed curve then there is a map
  $A_\ga : S^1\to S^1$ given by
$$A_{\ga}(\th) := \angle_p (\dot{\ga},X) \conj{ where } p = \ga(\th)$$
the angle between the tangent $\dot{\ga} = \ga_* (d/d\theta)$ of $\ga$ at $p$ and the vector $X = X_p$ at the same point.

Reinhart's {\em winding number} $\omega(\ga)\in\ZZ/\chi$ is the degree of the map $A_\ga$
taken modulo the Euler characteristic of $S$
$$ \omega(\ga)=\frac{1}{2\pi}\int_{S^1} A_\ga^* (d\theta) \quad\quad\imod {\chi}$$
when the vector field $X$ is chosen so that $\omega(a_i) = 0$ and $\omega(b_i) = 0$ for $\{a_i,b_i\}$ curves in the 1-skeleton which determines Eqn. \eqref{pipres}.
\end{defn}

Reinhart showed that this formula is independent of
the choice of the vector field $X$, within the imposed constraints, and the choice of the regular homotopy
representative of $\ga$ after the degree is taken modulo the Euler
characteristic of the surface \cite[Prop. 2]{reinhart}. This is due to the value of the integral
shifting by $\pm\chi$ as a regular homotopy moves a segment of $\ga$
over a zero of $X$.

\begin{remark}
Reinhart's definition is also independent of the chosen metric.
This is because the space of all
  Riemannian metrics on $S$ is path-connected, if the metric is
  continuously varied then the value of the integral varies continuously, 
  but since $\mathbb{Z}$ is discrete, the value must be constant.
\end{remark}

\subsection{Unobstructed curves and Abouzaid's lemma}\label{absec}
In this section we recall the notion of unobstructed curves and Abouzaid's
lemma which characterizes them in a convenient way.

\begin{defn}[Unobstructed curve]
Let $\tilde{S} \to S$ be the universal covering space of $S$. We say that an immersed loop $\a : S^1 \to S$ is {\em unobstructed} when it lifts to a properly embedded path in $\tilde{S}$.
\end{defn}

Abouzaid characterized unobstructed curves in terms of
fish-tails and contractability \cite{abouzaid}. Roughly speaking, a fish-tail is a homotopically trivial self-intersection, see Figure \ref{fig:fishtail}.

\begin{defn}[Fish-tail]\label{fishtaildef}
  If $\gamma : S^1\to S$ is an immersed curve then a {\em fish-tail} in $\gamma$ is a disk map
  $D^2\to S$ which maps the boundary of the disk to $\gamma$ and is non-singular
  everywhere except one point on $\ga$.
\end{defn}

Abouzaid's lemma states that for non-nulhomotopic curves fish-tails
represent the only obstruction to being unobstructed.

\begin{lemma}[Abouzaid]\label{abouzaidlem}
  A properly immersed smooth curve $\gamma$ is unobstructed if and only if
  it is not nulhomotopic and does not bound an immersed fish-tail.
\end{lemma}

Unobstructed curves are regular homotopy representatives of curves with no
unnecessary winding. The unobstructed loops are important in what follows because they will be our canonical
choice of regular representative within each free homotopy class.
More precisely, if $\hat{\pi}_{U,R} (S)$ denotes the set of regular free homotopy
classes of unobstructed curves on $S$ then we will show that there is a
bijection between the sets
$$\Phi : \hat{\pi}_{U,R} (S)\xto{\sim} \hat{\pi}(S)\backslash \{C\}$$
where $C$ denotes the free homotopy class of the contractible curve.  We
must remove $C$ from the righthand side since every regular representative
of the contractible free homotopy class $C$ is obstructed.

The first step in the construction of the grading will be to define the
grading on unobstructed curves in $\hat{\pi}_{U,R}(S)$. Since $C$ is contained in
the center of the Goldman Lie bracket, we can define the grading on the
Goldman Lie algebra modulo the center. Then the grading can be lifted by
choosing any regular homotopy representative of $C$ in $\pi_R(S)$. We will
choose the one in Figure \ref{fig:trivialrep}.

\begin{figure}[h]
\centering
\begin{center}
$$  \includegraphics[scale=0.75]{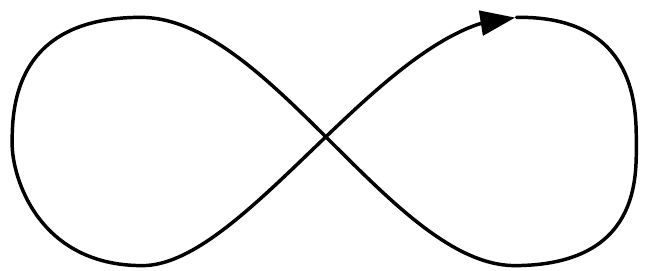}$$
\end{center}
  \caption{Our preferred choice of the regular representative for the free homotopy class $C$ of the contractible loop.}
  \label{fig:trivialrep}
\end{figure}

\section{Relations between free and regular homotopy}\label{constructionsec}
In this section we study the relationship between free and regular homotopy
classes of curves on a surface. In Section \ref{bundlemapssec}, we prove
that the map $s_* : \hat{\pi}_R(S) \to \hat{\pi}(S)$ which passes from
regular homotopy classes to free homotopy classes is a surjective, oriented
loop product preserving map. The fiber of $s_*$ over any free homotopy class
is identified with the $\ZZ$-orbit of any regular representative. In
\S\ref{bijectionsec}, a section $\Phi : \hat{\pi}(S) \to \hat{\pi}_{R}(S)$
is constructed by mapping free homotopy classes to unobstructed
representatives. This information is summarized by the diagram below.

\begin{equation*}\begin{tikzpicture}[scale=10, node distance=2.5cm]
\node (X) {$\hat{\pi}_R(S)$};
\node (Y) [below=1.5cm of X] {$\hat{\pi}(S)$};
\node (Z) [left=1cm of X] {$\ZZ\cdot \Phi$};
\node (W) [right=1.5cm of X] {};
\node (WW) [below=.5cm of W] {$\hat{\pi}_{U,R}(S) \cup \Phi(C)$};
\draw[->>,swap] (X) to node {$s_*$} (Y);
\draw[right hook->>,swap] (Y) to node  {$\Phi$} (WW);
\draw[right hook->] (Z) to node   {} (X);
\draw[left hook->] (WW) to node   {} (X);
\end{tikzpicture}
\end{equation*}

\subsection{Regular curves as a $\ZZ$-torsor}\label{bundlemapssec}
Here we show that the map $s_* : \hat{\pi}_R(S) \to \hat{\pi}(S)$, which
takes a free regular curve to its free homotopy class, is surjective
and preserves oriented loop products. The map $s_*$ commutes with a $\ZZ$-action on $\hat{\pi}_R(S)$ which is given by gluing fish-tails onto regular curves and there is $\ZZ$-equivariant isomorphism
from  the fiber over any free homotopy class to the $\ZZ$-orbit of any representative
$$\ZZ\cdot [\ga]_R \xto{\sim} s_*^{-1}([\ga]) \conj{ for } \ga\in[\ga].$$

\begin{defn}{($s_*$)}\label{sdef}
Composing the maps in Eqn. \eqref{smaleeqn} with the bundle map $s : S(TS)\to S$ defines a map
$$\pi_R(S,p) = \pi_0 (Imm_p(S^1,S))\cong \pi_1 (S(TS),p')\xto{\pi_1(s)} \pi_1(S,q),$$
which takes a based regular homotopy class of immersed curve
to its corresponding homotopy class.
This induces a map between sets of conjugacy classes
$$s_* : \hat{\pi}_R(S) \to \hat{\pi}(S).$$
This map is onto because every loop in $\pi_1(S,q)$ lifts to a loop in $\pi_1(S(TS),p')$; alternatively, $Imm_p(S^1,S)\subset Map(S^1,S)$ is dense. See also Lem. \ref{conjlemma}.
\end{defn}

\begin{defn}{($\ZZ$-action)}\label{zzactiondef}
There are maps $t, t^{-1} : \hat{\pi}_R(S)\to \hat{\pi}_R(S)$ which generate a $\ZZ$-action.
If $\ga$ is an immersed curve in $S$ then picking any point $p$ on $\ga$ we can add a small positive fish-tail at $p$ or a small negative fish-tail at $p$. In either case, this defines a new immersed curve which we call $t\ga$ or $t^{-1}\ga$. These operations are well-defined and independent of $p$ because we can both shrink a fish-tail to be arbitrarily small and drag it around the curve by regular homotopy.
The maps $t$ and $t^{-1}$ are illustrated by
$$\CPPic{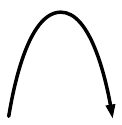} \mapsto\quad\CPPic{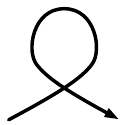} \conj{ and } \CPPic{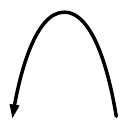} \mapsto\quad\CPPic{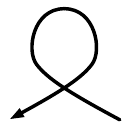}$$
respectively. These maps satisfy $tt^{-1} = 1_{\hat{\pi}_R(S)}$ and $t^{-1}t=1_{\hat{\pi}_R(S)}$ because, up to regular homotopy, a positive fish-tail can be cancelled with a negative fish-tail. One of the two cases is shown below.
$$\CPPic{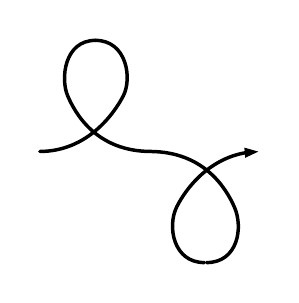} \hspace{.5in}\to \CPPic{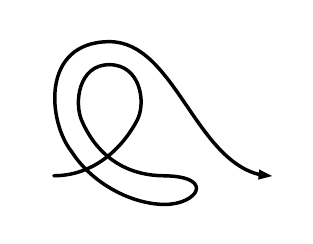}\hspace{.5in}\to \CPPic{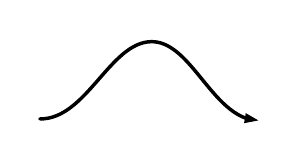}$$
The picture proof of the other case is obtained by reversing the orientation above.
So the maps $t$ and $t^{-1}$ define a $\ZZ \cong \inp{t}$ group action on $\hat{\pi}_R(S)$.

Since any additional fish-tails on a curve $\ga$ can be removed by homotopy, the $\ZZ$-action preserves the fiber 
$$s_*(t[\ga]_R) = s_*([\ga]_R)\conj{ and } s_*(t^{-1}[\ga]_R) = [\ga]_R.$$
Moreover, Reinhart's integral formula shows that this shifts the winding number, $\omega(t[\ga]_R) = \omega([\ga]_R) + 1$,
compare to Cor. \ref{fishtailpm1}. 
\end{defn}

\begin{remark}\label{smaleisormk}
By virtue of Smale's isomorphism $\phi : \hat{\pi}_R(S)\xto{\sim} \hat{\pi}(S(TS))$. There is an alternative way to define the $\ZZ$-action. Recall from Eqn. \eqref{spres} that there is a central element  $f \in \pi_1(S(TS),p')$ representing the fiber of the unit tangent bundle. If $[\ga]_R\in\hat{\pi}_R(S)$ then 
$$\phi(t[\ga]_R) = f\phi([\ga]_R) \in \hat{\pi}(S(TS))$$
so that multiplication by $f$
determines the action of $t$ on the set $\hat{\pi}_R(S)$. This can be seen directly by observing that a curve $t\ga$ with fish-tail in $Map(S^1,S(TS))$ is homotopic to the curve $\ga$ in $S$ with section vectors forming a positive full twist.
\end{remark}

\begin{defn}
The category of pointed sets and pointed set maps is the undercategory $Set_* := \{pt\}/Set$.
In more detail, an object is a pair $(A,a)$ consisting of a set $A$ and an element $a\in A$. A map $f : (A,a) \to (B,b)$ between two such pairs is a set map $f : A \to B$ which satisfies $f(a) = b$. A sequence of such maps
$$(A,a) \xto{f} (B,b) \xto{g} (C,c)$$
is {\em exact} when $g^{-1}(c) = im(f)$.  There is a terminal object $1 := (\{pt\},pt)$.
\end{defn}

The next lemma shows that passing from groups to conjugacy classes of
elements in groups is a functor from groups to pointed sets. The important
part of the lemma shows that a central inclusion of groups is mapped to an
injective map of pointed sets. So the functor preserves the exactness of
central extensions in a strong sense.

\begin{lemma}\label{conjlemma}
Taking conjugacy classes determines a functor $\hat{\cdot} : Grp \to Set_*$ from groups to pointed sets. In particular, a homomorphism $a : K \to G$ of groups induces a map $\hat{a} : \hat{K} \to \hat{G}$ from conjugacy classes in $K$ to conjugacy classes in $G$. This functor satisfies the three properties below.

\begin{enumerate}
\item A short exact sequence of groups induces a short exact sequence of pointed sets.
\item If $b : G \to H$ is an epimorphism then the set map $\hat{b}$ is surjective.
\item If $a : K \to G$ is a monomorphism and central, i.e.  $im(a) \subset Z(G)$, then the set map $\hat{a}$ is injective. 
\end{enumerate}
\end{lemma}
  \begin{proof}
If $a : K \to G$ then set $\hat{a}([k]) := [a(k)]$ if $k\sim k'$ so $k = rk'r^{-1}$ then $a([k]) = [a(rk'r^{-1})] = [a(r)a(k')a(r)^{-1})] = [a(k')]$. This is a functor, if $a : K \to G$ and $b : G \to H$ then $\hat{b}(\hat{a}([k])) = \hat{b}([a(k)]) = [ba(k)] = \widehat{ba}([k])$.

{\it 1.} Suppose $1 \to K \xto{a} G \xto{b} H \xto{p} 1$ is a short exact sequence of groups. Then there is a sequence of maps 
$$ 1 \to \hat{K} \xto{\hat{a}} \hat{G} \xto{\hat{b}} \hat{H} \xto{\hat{p}} 1$$
First we show that exactness holds on the left. If $[k] \in \hat{a}^{-1}([1])$ then $[a(k)] = [1]$ implies $a(k) = r1r^{-1} = 1$ for some $r\in G$. So $k\in ker(a) = \{1\}$. Therefore, $\hat{a}^{-1}([1]) = [1]$.

Exactness hold on the right because, for all $[h] \in \hat{p}^{-1}([1])$, there is a $g\in G$ such that $b(g) = h$ by surjectivity of $b$. So $\hat{b}([g]) = [b(g)] = [h]$. This is a special case of {\it 2.} below.

Now we will show that $im(\hat{a}) = \hat{b}^{-1}([1_H])$. First $\hat{b}\hat{a}[k] = [b(a(k))]= [1_H]$ implies $im(\hat{a}) \subset \hat{b}^{-1}([1_H])$. Suppose $[g]\in \hat{b}^{-1}([1_H])$, so $\hat{b}([g]) = [1_H]$ then there is $h\in H$ such that $hb(g)h^{-1} = 1_H$ or $b(g) = 1_H$, so $g\in ker(b) = im(a)$ and there exists $k\in K$ such that $a(k) = g$, it follows that $\hat{a}([k]) = [g]$. Thus $\hat{b}^{-1}([1_H]) \subset im(\hat{a})$, so $im(\hat{a}) = \hat{b}^{-1}([1_H])$.

{\it 2.} Assume that $b : G\to H$ is a surjective homomorphism. If $[h]\in \hat{H}$ then for any $h\in [h]$, there is a $g\in G$ such that $b(g) = h$. So $\hat{b}([g]) = [b(g)] = [h]$. 

{\it 3.} Assume that $a : K\to G$ is an injective homomorphism and $im(a) \subset Z(G)$. If $\hat{a}([k]) = \hat{a}([k'])$ then $[a(k)] = [a(k')]$. So there is an $r\in G$ such that $a(k) = ra(k') r^{-1}$, but $a(k') \in im(a) \subset Z(G)$ implies $a(k) = a(k')rr^{-1} = a(k')$, so $k=k'$ by injectivity of $a$.
\end{proof}

\begin{prop}\label{eqprop}
  Suppose that $[\tau]_R \in s_*^{-1}([\ga])$ is any curve in the fiber of $s_*$ over $[\ga]\in\hat{\pi}(S)$. Then there is a $\ZZ$-equivariant bijection
  $$\kappa : \ZZ\cdot [\tau]_R \xto{\sim} s_*^{-1}([\ga])$$
from the $\ZZ$-orbit of $[\tau]_R$ to the fiber.
  \end{prop}
\begin{proof}
This is because the kernel of $s_*$ is $\inp{f}\cong \ZZ$ where $f$ is the fiber.
In more detail,  there is a diagram
\begin{equation*}\begin{tikzpicture}[scale=10, node distance=2.5cm]
\node (A) {$1$};
\node (B) [right=2cm of A] {$\pi_1(S^1)$};
\node (C) [right=2cm of B] {$\pi_1(S(TS))$};
\node (D) [right=2cm of C] {$\pi_1(S)$};
\node (E) [right=2cm of D] {$1$};
\node (Y) [below=.75cm of C] {};
\node (X) [right=.75cm of Y] {$\pi_R(S)$};

\node (B2) [below=2cm of B] {$\hat{\pi}(S^1)$};
\node (C2) [below=2cm of C] {$\hat{\pi}(S(TS))$};
\node (D2) [below=2cm of D] {$\hat{\pi}(S)$};
\node (A2) [left=2cm of B2] {$1$};
\node (E2) [right=2cm of D2] {$1$};
\node (X2) [below=2cm of X] {$\hat{\pi}_R(S)$};

\draw[->] (X) to node {$\phi_1$} (C);
\draw[->] (X) to node [swap] {} (D);

\draw[->] (X) to node {} (X2);

\draw[->] (X2) to node {$\phi$} (C2);
\draw[->] (X2) to node [swap] {$s_*$} (D2);

\draw[->] (B) to node {} (B2);
\draw[->] (C) to node [swap] {$z$} (C2);
\draw[->] (D) to node {} (D2);

\draw[->] (A) to node {} (B);
\draw[->] (B) to node {} (C);
\draw[->] (C) to node {$\pi_1(s)$} (D);
\draw[->] (D) to node {} (E);

\draw[->] (A2) to node {} (B2);
\draw[->] (B2) to node {$\hat{\pi}(i)$} (C2);
\draw[->] (C2) to node {$\hat{\pi}(s)$} (D2);
\draw[->] (D2) to node {} (E2);
\end{tikzpicture}
\end{equation*}

The long exact sequence of homotopy groups associated to the circle bundle
$S^1 \xto{i} S(TS) \xto{s} S$ gives a short exact sequence among fundamental
groups because $S^1$ and $S$ are both $K(G,1)$-spaces. This gives the first row. The second row 
comes from Lem. \ref{conjlemma} above. The map $\phi$ is the Smale isomorphism and the
bottom triangle commutes because $s_* = \hat{\pi}(s)\circ\phi$. There is an isomorphism
$\pi_1(S^1) \cong \hat{\pi}(S^1) = \inp{f}\cong \ZZ$ is generated by the fiber element in the
Seifert presentation of $\pi_1(S(TS))$, see Eqn. \eqref{spres}.

  The map $\kappa$ is determined by $\ZZ$-equivariance and the assignment $0\cdot[\tau]_R \mapsto [\tau]_R$.
The two statements below are equivalent to injectivity and surjectivity of $\kappa$ respectively.
  \begin{enumerate}
    \item if $[x]_R\in s_*^{-1}([\ga])$ then $[x]_R= t^n[\tau]_R$ in $\pi_R(S,p)$ for some $n\in \ZZ$.
  \item $t^n[\tau]_R \not\simeq [\tau]_R$ for all $n\in \ZZ\backslash\{0\}$ and 
\end{enumerate}

For the first statement, fix a basepoint $p = (x(\theta_0), \dot{x}(\theta_0))$ on $x$ for an $x\in [x]_R$ and $\theta_0\in S^1$. In doing so $\phi[x]_R$ lifts trivially to $[x]_{\pi_1} \in z^{-1}(\phi[x]_R)$. Since $[x]_R\in s_*^{-1}([\ga])$, $\phi[x]_R \in \hat{\pi}(s)^{-1}([\ga])$ and $[x]_{\pi_1} \in \pi_1(s)^{-1}[\ga]$. But $\pi_1(s)^{-1}[\ga] = [\ga] ker(\pi_1(s)) = [\ga]\inp{f}$. So $[x]_{\pi_1} = f^n [\ga]$ for some $n\in\ZZ$. The rest follows from Rmk. \ref{smaleisormk}.

  The second statement follows from the injectivity of $\hat{\pi}(i)$ and centrality of $f\in \pi_1(S(TS),p')$. Here is a detailed argument. By Rmk. \ref{smaleisormk} we may identify the action of $t$ with multiplication by the fiber element $f$.
  Notice that if the genus of $S$ is one or $[\tau]_R = f^NC$ is homotopically trivial then the statement is trivial. So assume that the genus is at least two,  $\tau$ is homotopically non-trivial and the conjugacy classes 
  $f^n[\tau]_R = [\tau]$ are equal in $\hat{\pi}(S(TS))$ for some $n\in \ZZ$. There is a curve $\a$ such that  $f^n\tau = \a\tau\a^{-1}$ or $f^n = \a\tau\a^{-1}\tau^{-1}$. Applying the map $s':=\pi_1(s)$ gives $1 = s'(\a\tau\a^{-1}\tau^{-1})$ in $\pi_1(S)$. In particular, $s'(\a) \in Cent_{\pi_1}(s'(\tau))$. The centralizer of a non-trivial element $s'(\tau)\in \pi_1(S)$ is cyclic $\inp{\b}$ when $S$ is hyperbolic, which is true when the genus is greater than two, see \cite[\S 1.1.3]{fm}. Every element $s'(\tau)$ is contained in it's own centralizer. So
  $$s'(\a), s'(\tau) \in Cent_{\pi_1}(s'(\tau))=\inp{\b}$$
and $s'(\a) = \b^q$ and $s'(\tau) = \b^m$
  for some $q, m\in \ZZ$ and for some $\b \in \pi_1(S)$. By exactness of the first row, there are integers $p,\ell\in\ZZ$ such that $\a = f^p \b^q$ and $\tau = f^\ell \b^m$. Combining these observations gives
\begin{align*}
  f^n &= \a\tau\a^{-1}\tau^{-1}\\
  &= f^p\b^{q} f^\ell b^m f^{-p} \b^{-q} f^{-\ell} b^{-m}\\
  &= \b^{q} b^m \b^{-q} b^{-m}\\
  &= 1
\end{align*}
in $\pi_1(S(TS))$, which contradicts injectivity of the fiber map $\pi_1(i)$.

\end{proof}

\begin{rmk} The map $\kappa$ depends on the choice $\tau$. The section $\Phi$ in \S\ref{bijectionsec} will give a canonical choice of $\tau$. \end{rmk}

\begin{prop}\label{orloopprop}
  The map $s_*$ preserves the oriented loop product
$$s_*([\a\cdot_p\b]_R) = s_*([\a]_R) \cdot_q s_*([\b]_R).$$
  \end{prop}
\begin{proof}
  Since $s_*$ is induced by a map $\pi_1(S(TS),p')\to \pi_1(S,q)$ between
fundamental groups, it must preserve the product at any basepoint $p$. 
\end{proof}

\subsection{The $s_*$-section $\Phi$}\label{bijectionsec}

Here we define a section $\Phi : \hat{\pi}(S)\to \hat{\pi}_R(S)$ of the map $s_* : \hat{\pi}_R(S) \to \hat{\pi}(S)$. The map $\Phi$ is injective and its image is the set 
$$\im(\Phi) = \hat{\pi}_{U,R}(S)\cup \Phi(C)$$
of unobstructed regular curves in $S$ together with a choice $\Phi(C)$ of regular representative for the contractible curve (see Fig. \ref{fig:trivialrep}). 
These properties are established in Proposition \ref{correspprop}.  
The proof of this proposition uses Lemma \ref{existencelemma}
which shows that every free homotopy class contains an unobstructed
representative and Lemma \ref{uniquenesslemma} which shows that any two
unobstructed representatives of the same free homotopy class are freely
regularly homotopic.

\begin{lemma}\label{existencelemma}
Let $\alpha:S^1\to S$ be a non-nulhomotopic loop in $S$ and let $[\alpha]\in \hat{\pi}(S)\backslash\{ C\}$ denote its free homotopy class. Then there is a representative $\bar{\alpha}\in [\alpha]$ such that $\bar{\alpha}$ is unobstructed.
\end{lemma}

\begin{proof} 
  Let $[\alpha]\in \hat{\pi}(S)\backslash \{C\}$ with $\alpha$ a normal
  representative. Assume by contradiction that $\alpha$ fails to be
  unobstructed. Since $\a$ is not nulhomotopic, Abouzaid's Lemma \ref{abouzaidlem} implies that
  the image of $\alpha$ bounds a finite (by normality) number of fish-tails. By Definition \ref{fishtaildef} of fish-tail, there must be a disk through which any fish-tail
can be contracted. Such disks determine a smooth convex homotopy
  $F:S^1 \times [0,1] \to S$ with support in a contractible neighborhood of
  the fish-tails that straightens them from innermost to
  outermost into an immersed embedded line segment. 
  Thus the resulting loop $\bar{\alpha}$ representing $[\alpha]$ obtained
  via $F$ is regular, non-trivial, has no fish-tails, and is unobstructed, again,
  by Abouzaid's Lemma \ref{abouzaidlem}.
\end{proof}

The lemma above establishes the existence of unobstructed representatives of
curves within each free homotopy class. The next lemma will show that there is a
unique choice up to regular homotopy. We will use a relative version of Epstein's
Theorem which we now recall.

\begin{thm}{(\cite{Epstein})}\label{epstein}
Let $S$ be a surface with boundary and $\a$, $\b$ two embedded arcs with endpoints that are equal and contained in the boundary of $S$.  If $\a$ and $\b$ are homotopic with endpoints fixed then $\a$ and $\b$ are isotopic with endpoints fixed.
  \end{thm}

We are now ready to prove our lemma.

\begin{lemma}\label{uniquenesslemma} 
If two unobstructed curves $\alpha$ and $\beta$ are freely homotopic then they are freely regularly homotopic.
\end{lemma}

The interplay between different types of equivalence classes of loops now becomes important.
Recall our notational convention from Section \ref{surfacetopsec} of appending subscripts to brackets around loops, indicating the quotient set to which the equivalence class belongs.

\begin{proof}
 Let $\alpha$ and $\beta$ be unobstructed curves belonging to the same free homotopy class: $[\alpha]_{\hat{\pi}}=[\beta]_{\hat{\pi}}$, given in general position, in particular there are finitely many points of  intersection.

 First we show that we can reduce the problem to the case where $\alpha$ and
 $\beta$ are homotopic at a fixed basepoint. Since $\alpha$ and $\beta$ are
 freely homotopic, there are basepoints $a$ and $b$ for $\a$ and $\b$ and a
 path $\gamma:[0,1]\to S$ from $a$ to $b$ such that
 $$[\alpha]_{\pi_1(S,a)}=[\bar{\gamma}\cdot\beta\cdot\gamma]_{\pi_1(S,a)}$$
 After a small homotopy we can assume that the curves $\gamma$ and
 $\bar{\gamma}$ are smooth, regularly parametrized, concatenate regularly
 with each other and $\beta$ and have no fish-tails. In particular, the
 cyclic segment $\gamma\cdot\bar{\gamma}$ can be made to share the same
 tangent vector as $\alpha$ at $a$ and to freely regularly retract back to
 the original path segment of $\beta$ crossing $b$.  In this way we obtain a
 new unobstructed representative $\tilde{\beta}$ of
 $[\bar{\gamma}\cdot\beta\cdot\gamma]_{\pi_1(S,a)}$ which is regularly
 freely homotopic to $\beta$ and basepoint homotopic to $\alpha$.

So, without loss of generality, we make the additional assumption that
 $\alpha$ and $\beta$ share a basepoint and are basepoint homotopic
\begin{equation}\label{abhtyeq} [\alpha]_{\pi_1}=[\beta]_{\pi_1}.\end{equation}
Let $\pi :\tilde{S}\to S$ be the universal covering and, after fixing a
basepoint $\tilde{p}\in\pi^{-1}(p)$, there are paths $\hat{\alpha}$ and
$\hat{\beta}$ lifting the loops $\alpha$ and $\beta$ respectively.  Since
$\a$ and $\b$ are unobstructed, the paths $\hat{\a}$ and $\hat{\b}$ are
embedded. Equation \eqref{abhtyeq} implies that these
lifts share endpoints and are homotopic rel endpoints.  So by Epstein's Theorem \ref{epstein},
$\hat{\a}$ and $\hat{\b}$ are isotopic rel endpoints in $\tilde{S}$.
Applying the covering map $\pi$ to this isotopy produces a regular homotopy
between $\alpha$ and $\beta$.
\end{proof}

Combining Lemma \ref{existencelemma} and Lemma \ref{uniquenesslemma} shows
that any two choices of unobstructed representatives for a free homotopy
class are freely regularly homotopic. Thus any non-nulhomotopic free
homotopy class $[\ga]$ contains a unique unobstructed representative $\Phi([\ga])$ up to free
regular homotopy. This is recorded by the proposition below.

\begin{prop}\label{correspprop}
There is a bijection $\Phi$ between non-nulhomotopic free homotopy classes of curves on $S$ and regular free homotopy classes of unobstructed curves on $S$ 
$$\Phi : \hat{\pi}(S)\backslash \{C\} \xto{\sim}  \hat{\pi}_{U,R} (S)$$
such that a free homotopy class $[\alpha]$ contains the unobstructed regular homotopy class to which it corresponds.
\end{prop}
 \begin{proof} 
   By Lemma \ref{existencelemma}, for any non-nulhomotopic free homotopy
   class $[\alpha]_{\hat{\pi}}$, there exists an unobstructed representative
   $\alpha_U\in[\alpha]_{\hat{\pi}}$. Furthermore, by Lemma
   \ref{uniquenesslemma}, any other unobstructed representative
   $\alpha_U ^\prime\in[\alpha]_{\hat{\pi}}$ is freely regularly homotopic
   to $\alpha_U$. Thus
   $[\alpha_U]_{\pi_{U,R}}=[\alpha_U ^\prime]_{\pi_{U,R}}$. So we define
$$\Phi([\alpha]_{\hat{\pi}}):=[\alpha_U]_{\pi_{U,R}}.$$

Now the map $\Phi$ is injective because if two loops $\beta$ and $\gamma$
are not freely homotopic then any choice of unobstructed representatives
$\beta_U$ and $\gamma_U$ will, by transitivity, not be freely homotopic, and so
cannot be freely regularly homotopic.  The map $\Phi$ is surjective because any
unobstructed curve is non-nulhomotopic and thus belongs to some
non-nulhomotopic free homotopy class.
\end{proof}

In order to define the grading, we must address
additivity under the product $\alpha\cdot_p \beta$ through
this correspondence. This is the content of the next section.

\section{Grading the Goldman Lie algebra}\label{windgoldsec}

We want to use the map $\Phi$ in Proposition \ref{correspprop} to define a
grading on the Goldman Lie algebra.  In Lemma \ref{prodlem} we show that the
oriented loop product of two non-nulhomotopic unobstructed curves at a point of intersection
is unobstructed when the point of intersection is not the corner of a
bigon. In Proposition \ref{wadditiveprop} we show that the winding number is
additive over unobstructed curves.  The section concludes with the grading
on the Goldman Lie algebra in Theorem \ref{gsgrthm}.

\begin{lemma}\label{prodlem} 
  Suppose $\alpha$ and $\beta$ are unobstructed curves intersecting at a
  point $p\in S$. If $p$ is not the corner of a bigon then the oriented loop
  product $\alpha\cdot_p\beta$ is either unobstructed or nulhomotopic.
\end{lemma}
\begin{proof}
  Let $U$ be a small neighborhood of the point $p$ as in Figure
  \ref{fig:loopprod}. By contrapositive, assume $\alpha\cdot_p\beta$ is
obstructed and not nulhomotopic,
 by Abouzaid's Lemma
  \ref{abouzaidlem}, the product $\alpha\cdot_p \beta$ contains a
  fish-tail. Observing that a self-intersection must occur at a fish-tail,
  and that $\alpha\cdot_p \beta$ does not self-intersect in $U$ by
  construction, we have that there must be some point of self-intersection
  $q\in S \backslash U$. In $S\backslash U$, connected components
  of $\alpha\cdot_p\beta$ belong originally to either $\alpha$ or $\beta$
  which by property of being unobstructed, cannot have fish-tails. Thus the
  fish-tail beginning at $q$ must cross through $U$. Since the fish-tail at
  $q$ bounds a disk, by undoing the resolution at $p$ we can see that this
  implies that $p$ and $q$ are the corners of a bigon as in Figure
  \ref{fig:bigonbegone} below.
\end{proof}

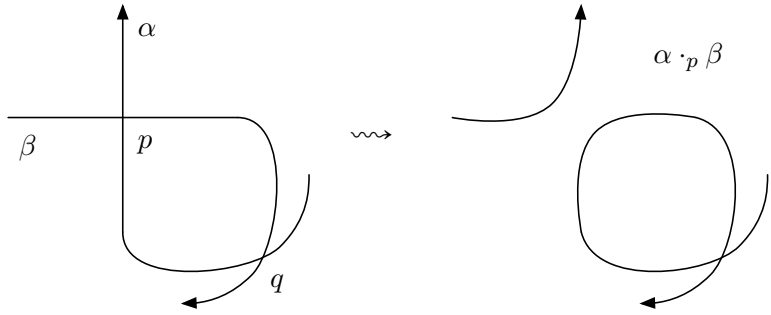
\begin{figure}[h]
\begin{center}
\begin{overpic}[scale=0.6]
{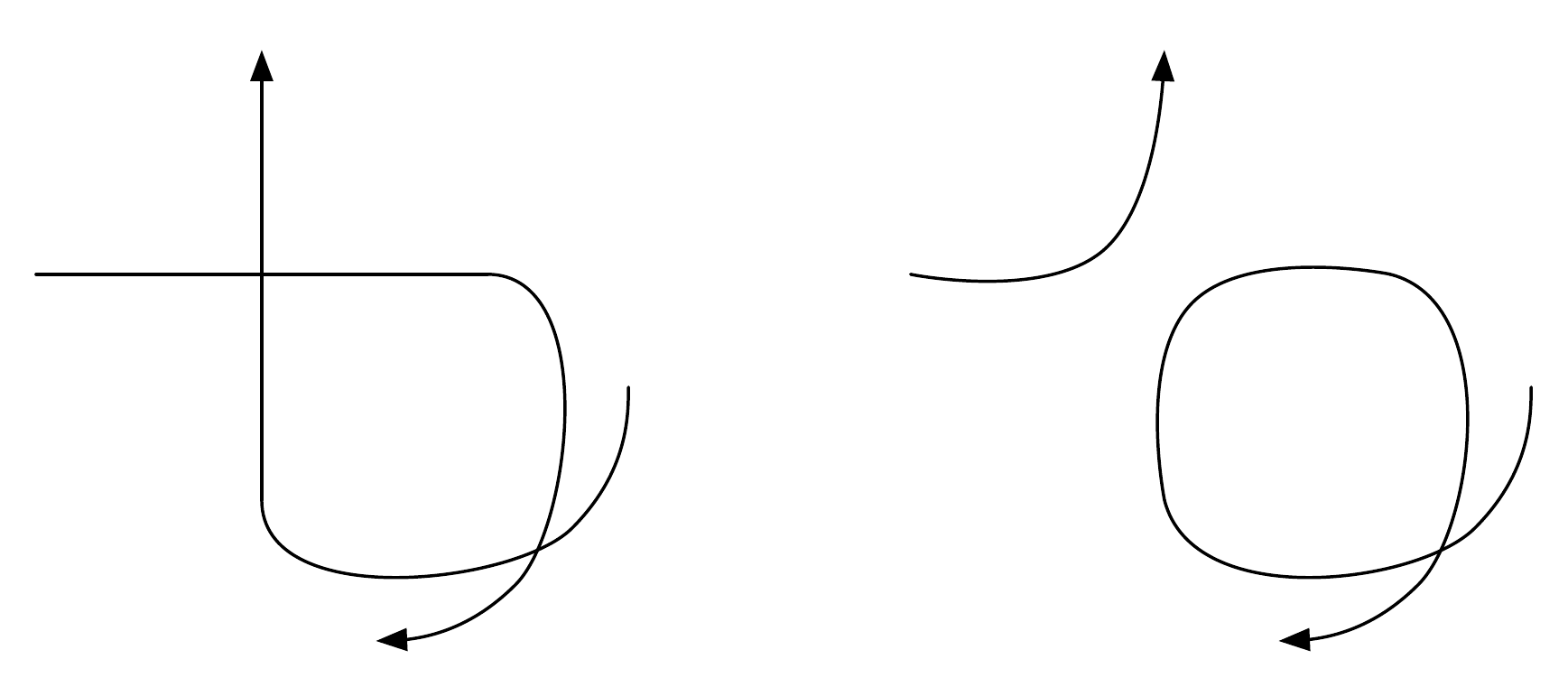}
\put(135,70){$\rightsquigarrow $}
\put(55,110){$\a$}
\put(10,65){$\b$}
\put(250,100){$\a\cdot_p\b$}
\put(55,67.5){$p$}
\put(105,15){$q$}
\end{overpic}
  \caption{Recovering a bigon from a fish-tail that crosses $U$}
  \label{fig:bigonbegone}
\end{center}
\end{figure}

\begin{cor}\label{phihom}
  If $\ga$ and $\tau$ do not represent the trivial nulhomotopic class and $p$ is not the corner of a bigon then
  $$\Phi([\ga\cdot_p \tau]) = \Phi([\ga])\cdot_p\Phi([\tau])$$
\end{cor}
  \begin{proof}
The map $\Phi$ contracts fish-tails. The equation holds when no new fish-tails are created by the loop product and Lem. \ref{prodlem} above gives this criteria.
    \end{proof}

\begin{rmk}\label{NWRK}
Note that if $p$ and $q$ are the corners of a bigon between two curves $\alpha$ and $\beta$ then 
$$[\alpha \cdot_p \beta]=[\alpha \cdot_q \beta]
\quad\quad\text{ and }\quad\quad
\epsilon_p(\alpha,\beta)=-\epsilon_q(\alpha,\beta).$$ So the two summands in
$[[\alpha],[\beta]]$ corresponding to $\alpha\cdot_p\beta$ and
$\alpha\cdot_q \beta$ cancel with each other. Thus for our purposes, when
thinking about the Goldman Lie bracket, we do not need to concern ourselves
with additivity of the winding number of curves over oriented resolution at
the corners of a bigon.
\end{rmk}

We next show that $\omega$ is additive when the criteria above is satisfied.

\begin{prop}\label{wadditiveprop}
Let $\alpha$ and $\beta$ be unobstructed curves. Suppose that $p\in\alpha\cap\beta$ is not the corner of a bigon between $\alpha$ and $\beta$ then the winding number is additive over the oriented loop product
 $$ \omega(\alpha\cdot_p\beta)=\omega(\alpha)+\omega(\beta). $$
\end{prop}

\begin{proof}
The curves $\a$ and $\b$ determine classes $[\a]_R$ and $[\b]_R$ and
 the oriented loop product $(\a,\b)\mapsto \a\cdot_p\b$ is the product in
  the group $\pi_R(S,p)$.  In Section \ref{windingsec}, the winding number
  $\omega : \pi_R(S,p) \to \ZZ/\chi$ is defined to be a homomorphism on this
  group. If $p$ is not the corner of a bigon then $\a\cdot_p\b$ will be
  unobstructed and so it will agree up to regular homotopy with the
  unobstructed representative which determines the winding number.
\end{proof}
A second geometric argument uses Reinhart's definition of winding number.
\begin{proof}
Arrange the curves as in Figure \ref{fig:loopprod} and let $U$ be a small neighborhood of $p$. 
We choose a vector field $X$ which is constant in the horizontal direction.
  Let $A_\alpha$,$A_\beta$ and $A_{\alpha\cdot_p\beta}$ be the angle functions
  from Definition \ref{RHdef}. We will prove that
\begin{equation}\label{inteq}
  \frac{1}{2\pi}\int_{S^1} A^*_\alpha (d\theta)+\frac{1}{2\pi}\int_{S^1} A^*_\beta (d\theta) =\frac{1}{2\pi}\int_{S^1} A^*_{\alpha\cdot_p\beta} (d\theta) \quad\quad\imod{\chi}.
  \end{equation}

Since the curves $\alpha\cup\beta$ and $\alpha\cdot_p\beta$ are identical
outside of $U$, the integrals take the same values when restricted to the
preimage of $S\backslash U$ in $S^1$.

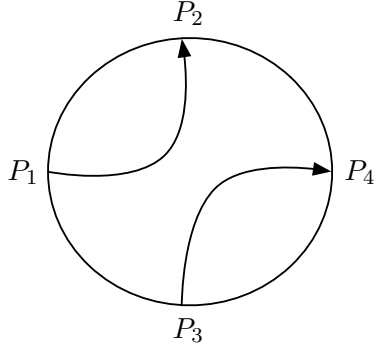
\begin{figure}[h]
\begin{center}
\begin{overpic}[scale=0.7]
{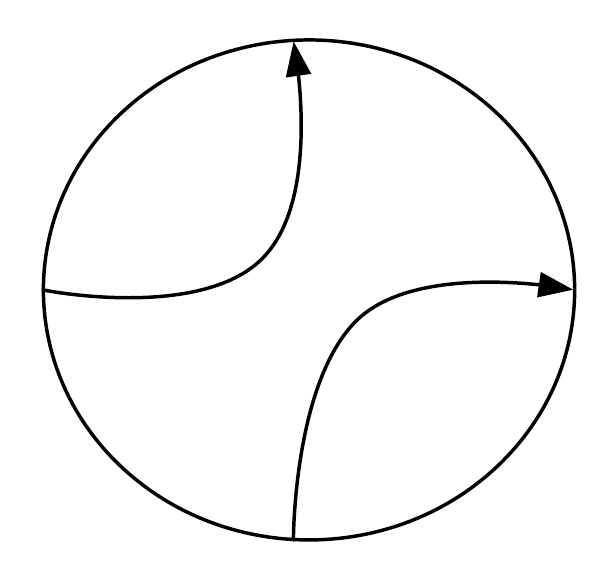}
\put(-7.5,55){$P_1$}
\put(120,55){$P_4$}
\put(55,-5){$P_3$}
\put(55,115){$P_2$}
\end{overpic}

   \caption{Labelled points of entry and exit of $\alpha\cdot_p\beta$ on $\partial\bar{U}$}
   \label{fig:additive}
\end{center}
\end{figure}

We must examine what happens within $U$. We show that the component of the
integrals inside the preimage of $U$ in $S^1$ provides no contribution to
the integrals on both sides of the equality. 
Without loss of generality, assume $X$ is tangent
to $\alpha\cdot_p\beta$ at $P_1$ and the rest of $X$ on $U$ is obtained by
parallel translation.
Then when traversing the curve $\alpha\cdot_p\beta$ from $P_1$ to $P_2$, there is a change
of $\pi/2$ in angle. Similarly traversing $\alpha\cdot_p\beta$ from $P_3$
to $P_4$, the opposite change of angles between the tangent vectors and $X$
occurs in a mirrored fashion, giving an overall contribution of zero to the
total change in angles between the curve's tangent and $X$ inside of
$U$. Since $\alpha$ is perpendicular to $X$ in $U$ and $\b$ is parallel to $X$ in $U$, 
the total change of angle between tangent vectors and $X$ in $U$ is zero for
both $\alpha$ and $\beta$. Hence Equation \eqref{inteq} holds.

\end{proof}

The proposition above implies the corollaries below.

\begin{cor}\label{rescor}
Crossings can be resolved without changing winding number of regular homotopy classes.
Crossings which do not bound bigons can be resolved without changing the winding number of free homotopy classes
$$\CPPic{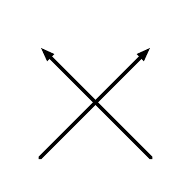} \sim_\omega \CPPic{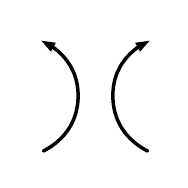}$$
where $\ga \sim_\omega \ga'$ when $\omega(\ga) = \omega(\ga')$.
\end{cor}

An important special case shows that fish-tails can be removed. 

\begin{cor}\label{fishtailpm1}
Fish-tail crossings can be resolved without changing the winding number
$$\CPPic{pcorres1} \sim_\omega\quad \CPPic{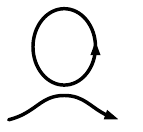} \conj{ and }\quad \CPPic{pcorres2} \sim_\omega \quad \CPPic{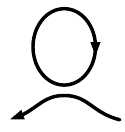}$$
where $\ga \sim_\omega \ga'$ when $\omega(\ga) = \omega(\ga')$.
\end{cor}

Since a counterclockwise nulhomotopic curve has winding number $+1$. The two
corollaries above show that the winding number of any curve can be computed
in terms of its unobstructed representative by removing bigons and
fish-tails.

We are ready to define the grading.

\begin{defn}[Grading on $\gs$]
Recall that $\gs = \ZZ\inp{\hat{\pi}(S)}$.  Let $[\alpha]$ be an element of $\hat{\pi}(S)\backslash \{C\}$. Let $\Phi([\a])$ be the regular homotopy class of the unobstructed representative of $[\a]$.  The {\em winding number} 
  $\omega_{\g}([\alpha])$ of $[\alpha]$
is defined to be 
  the winding number of $\Phi([\a])$. In the language of Proposition \ref{correspprop},
$$\omega_{\g}([\alpha]_{\hat{\pi}}):=\omega(\Phi([\alpha]_{\hat{\pi}})).$$
Following discussion in Section \ref{absec}, if $C$ represents the free homotopy class of the contractible curve then we choose $\Phi(C)$ to be the infinity curve and set $\omega_{\g}(\Phi(C)) = 0$.
\end{defn}

The theorem below confirms that this definition gives a cyclic $\ZZ/\chi$-grading on the Goldman Lie algebra.

\begin{theorem}\label{gsgrthm}
The map $\omega_{\g}$ defines a grading of Goldman Lie algebra $\gs$.
\end{theorem}
\begin{proof}
  Since the unobstructed representative of a free homotopy class of curves
  is well-defined up to regular free homotopy, so is its winding
  number. Furthermore, by Proposition \ref{wadditiveprop}, along with the fact that the sum
  expression (Def. \ref{goldmandef}) for $[[\alpha],[\beta]]$ simplifies to contain no
  contribution from bigon corners in $\alpha\cap\beta$, we have that
  $[[\alpha],[\beta]]$ only consists of summands that are oriented loop
  products at points where additivity holds and   thus
  $$\omega_{\g}([[\alpha],[\beta]])=\omega_{\g}([\alpha])+\omega_{\g}([\beta]).$$
  Hence the above map determines a grading of the Goldman Lie algebra $\gs = \ZZ\inp{\hat{\pi}(S)}$.
\end{proof}

Recall that when $G = GL(n,\RR), GL(n,\CC)$ or $GL(n,\HH)$, Goldman
constructed a Poisson action of $\gs$ on the representation variety
$Hom(\pi_1(S),G)/G$, \cite[Thm. 5.4]{goldman}.  We ask whether there is a
compatible grading on the representation variety.

\begin{question}
Is the action of the Goldman Lie algebra on representation varieties graded?
\end{question}

\section{Regular Goldman Lie Algebra}\label{regoldmansec}

Just as the Goldman Lie algebra $\gs$ is defined in terms of free homotopy
classes of curves $\freehtpy(S)$, there is a Lie algebra $\grs$ which is
defined in terms of free regular homotopy classes of immersed curves
$\reghtpy(S)$.

\begin{definition}\label{def:reggla}
The {\em regular Goldman Lie algebra} of a surface $S$ is the free abelian group on the set $\reghtpy(S)$ of free regular homotopy classes of immersed curves 
$$\grs :=\ZZ\inp{\reghtpy(S)}.$$
The Lie bracket is defined in precisely the same way as the Goldman Lie bracket in Def. \ref{goldmandef}.
Given $[\a]_R,[\b]_R\in \reghtpy(S)$, the bracket is 
\begin{equation}\label{reggoldmanbracket}
  [[\alpha]_R,[\beta]_R] := \sum_{p\in \alpha\cap\beta}\epsilon_p (\alpha,\beta)[\alpha\cdot_p\beta]_R.
  \end{equation}
\end{definition}

The bracket above determines a Lie algebra structure.  In fact, the same
argument Goldman used to show that $\g(S)$ is a Lie algebra also shows that
$\g_R(S)$ is a Lie algebra \cite[Thm.~5.3]{goldman}.

All of the acrobatics related to $\Phi$ and unobstructed regular
representatives, used to introduce the winding number grading on $\g(S)$, 
are unnecessary for the regular Goldman Lie
algebra.  This is because the winding number homomorphism
$\omega:\pi_{R}(S,p)\to\ZZ/\chi$ is defined on $\pi_R(S,p)$, so it defines a map
on conjugacy classes $\reghtpy(S)$ by Lem. \ref{conjlemma}.
The proposition below summarizes this discussion.

\begin{prop}
  The regular Goldman Lie algebra $\grs$ is graded by winding number.
\end{prop}

In Def. \ref{zzactiondef} a $\ZZ$-action on $\hat{\pi}_R(S)$ was introduced. Here we note
that it is compatible with the bracket above.
\begin{lemma}\label{equivlem}
  The regular Goldman Lie bracket is $\ZZ$-equivariant
  $$t[[\a]_R,[\b]_R] = [t[\a]_R,[\b]_R] = [[\a]_R,t[\b]_R]\conj{ for } t\in \ZZ.$$
  \end{lemma}
  \begin{proof}
A generator $t\in \ZZ$ acts on any term $[\alpha\cdot_p\beta]_R$ in Eqn. \eqref{reggoldmanbracket} by adding a fish-tail somewhere along the curve. This fish-tail can be smoothly and regularly isotoped to be small and then moved along the curve to either $\alpha$ or $\beta$. This commutes with the resolution in Fig. \ref{fig:loopprod}.
    \end{proof}

 The results of Section \ref{constructionsec}, relating the set of free
 homotopy classes of curves $\freehtpy(S)$ to the set of regular homotopy
 classes of curves $\reghtpy(S)$, can be extended to describe the
 relationship between the Goldman Lie algebra $\g(S)$ and the regular
 Goldman Lie algebra $\grs$.

 \begin{defn}
   Suppose $\g$ is a Lie algebra over a commutative ring $k$ and $R$ is a commutative $k$-algebra then $\g\ott_k R$ is a Lie algebra with product
\begin{equation}\label{exteqn}
  [x\ott r,y\ott s] := [x,y]\ott rs.
\end{equation}  
The {\em loop algebra} is 
$$\cL\g := \g \ott_k k[t,t^{-1}].$$
   \end{defn}

The loop algebra also has a canonical $\ZZ$-action generated by $t$; $t\cdot(x\ott t^n) = x\ott t^{n+1}$. This suggests a relationship between the loop algebra of the Goldman Lie algebra $\cL\gs$ and the regular Goldman Lie algebra $\grs$. The next theorem addresses this point.

\begin{thm}\label{reggoldmanthm}
There is a $\ZZ$-equivariant Lie algebra isomorphism
$$\a :  \cL\gs \xto{\sim} \grs$$
from the loop algebra of the Goldman Lie algebra to the regular Goldman Lie algebra.
\end{thm}   
\begin{proof}
There is a map $\a : \cL\gs \to \grs$ and an inverse map $\b : \grs \to \cL\gs$ which are determined by
$$\a([\ga] \ott t^n) := t^n\Phi([\ga]) \conj{ and } \b([\tau]_R) := s_*([\tau]_R)\ott t^{d([\tau]_R)}$$
where $d([\tau]_R)$, by Prop. \ref{eqprop}, is the unique function $d : \hat{\pi}_R(S)\to \ZZ$ which satisfies the equation
\begin{equation}\label{ddef}
  [\tau]_R = t^{d([\tau]_R)}\Phi(s_*([\tau]_R)).
\end{equation}
The value of $d([\tau]_R)$ is the number of fish-tails needed to make the unobstructed representative $\Phi(s_*[\tau]_R)$ regularly homotopic to $[\tau]_R$. In particular, $d(\Phi([\ga])) = 0$ for all $[\ga]\in\hat{\pi}(S)$.

The map $d$ is $\ZZ$-equivariant because 
$t^n[\tau]_R = t^{d([\tau]_R)+n}\Phi(s_*([\tau]_R)) = t^{d([\tau]_R)+n}\Phi(s_*(t^n[\tau]_R))$. Which gives
$d(t^n[\tau]_R) = d([\tau]_R) +n$ for $n\in\ZZ$ from Eqn. \eqref{ddef}. 

The map $\b$ is $\ZZ$-equivariant because $\b(t^n[\tau]_R) = s_*(t^n[\tau]_R)\ott t^{d(t^n[\tau]_R)} = s_*([\tau]_R)\ott t^{d(t^n[\tau]_R)} = s_*([\tau]_R)\ott t^nt^{d([\tau]_R)} = t^n\b([\tau]_R)$.

Also the map $\a$ is $\ZZ$-equivariant. This is because $\a(t^n\cdot([\ga]\ott t^m)) = \a([\ga]\ott t^{n+m}) = t^{n+m}\Phi([\ga]) = t^n(t^m\Phi([\ga])) = t^n\a([\ga]\ott t^m)$.

Next we show $\b\a = 1_{\cL\gs}$.
\begin{align*}
  \b\a([\ga]\ott t^n) &= \b(t^n\Phi([\ga])) & (\normaltext{Def. of } \a)\\
                      &= t^n\b(\Phi([\ga])) & (\normaltext{Equivariance})\\
                      &= t^n\cdot(s_*(\Phi([\ga])) \ott t^{d(\Phi([\ga]))}) & (\normaltext{Def. of } \b)\\
                      &= s_*(\Phi([\ga])) \ott t^{n+d(\Phi([\ga]))} & (\ZZ\normaltext{-action})\\
                      &= [\ga] \ott t^{n+d(\Phi([\ga]))} & (s_*\Phi=1)\\
                      &= [\ga] \ott t^n & (d(\Phi([\ga])) = 0)
  \end{align*}
where  $d(\Phi([\ga])) = 0$ is discussed above.

Now we show that $\a\b = 1_{\grs}$.
\begin{align*}
  \a\b([\tau]_R) &= \a(s_*([\tau]_R)\ott t^{d([\tau]_R)}) & (\normaltext{Def. of } \b)\\
                 &= t^{d([\tau]_R)}\Phi(s_*([\tau]_R)) & (\normaltext{Def. of } \a)\\
  &= [\tau]_R & \eqref{ddef} 
  \end{align*}

Next we show that the map $\a$ is a Lie algebra homomorphisms. 
If $[\ga]$ is contractible then $[[\ga],[\tau]]=0$ for any $[\tau]$, but $\a([\ga]\ott t^n) = t^n\Phi([\ga])$ so that $[\a([\ga]),\a([\tau])]=0$ too. So we assume that $[\ga]$ and $[\tau]$ are not nulhomotopic. In the equations below, all of the sums are taken over $p \in \ga\cap\tau$ and $\e_p = \e_p(\ga,\tau)$; so, without loss of generality, we assume that no $p$ is the corner of a bigon, see Rmk. \ref{NWRK} and Lem. \ref{dhom} below.
\begin{align*}
  \a([[\ga]\ott t^n,[\tau]\ott t^m]) &=\a(\Sigma_{p} \e_p [\ga\cdot_p\tau])\ott t^{n+m} & \eqref{goldmanbracket}\\ 
  &=t^{n+m} \Phi(\Sigma_{p} \e_p [\ga\cdot_p\tau]) & (\textnormal{Def. of } \a)\\
  &=t^{n+m} \Sigma_{p} \e_p \Phi[\ga\cdot_p\tau] & (\textnormal{Linearity})\\
  &=t^{n+m} \Sigma_{p} \e_p [\Phi([\ga])\cdot_p\Phi([\tau])] & (\textnormal{Cor. } \ref{phihom})\\
  &=t^{n+m} [\Phi([\ga]),\Phi([\tau])] & \eqref{goldmanbracket}\\
  &=[t^n \Phi([\ga]),t^m\Phi([\tau])] & (\textnormal{Lem. } \ref{equivlem})\\
&=[\a([\ga]\ott t^n),\a([\tau] \ott t^m)] & (\textnormal{Def. of } \a )
\end{align*}

The inverse $\b$ of a bijective Lie algebra homomorphism $\a$ is necessarily a Lie algebra homomorphism because $\b[x,y] = \b[\a\b(x),\a\b(y)] = \b\a[\b(x),\b(y)] = [\b(x),\b(y)]$. For completeness, we include a direct proof below.
\begin{align*}
  \b[[\ga]_R,[\tau]_R] &= \b(\Sigma_{p} \e_p [\ga\cdot_p\tau]_R) & \eqref{reggoldmanbracket}\\
  &=\Sigma_{p} \e_p \b[\ga\cdot_p\tau]_R & (\textnormal{Linearity})\\
  &=\Sigma_{p} \e_p s_*([\ga\cdot_p\tau]_R) \ott t^{d([\ga\cdot_p\tau]_R)} & (\textnormal{Def. of } \b)\\
&=\Sigma_{p} \e_p s_*([\ga]_R)\cdot_p s_*([\tau]_R) \ott t^{d([\ga]_R)+d([\tau]_R)} & (\textnormal{Prop. } \ref{orloopprop}\,\&\, \textnormal{Lem. } \ref{dhom})\\
&=[s_*([\ga]_R), s_*([\tau]_R)]\ott t^{d([\ga]_R)+d([\tau]_R)} & \eqref{goldmanbracket} \\
&=[s_*([\ga]_R)\ott t^{d([\ga]_R)}, s_*([\tau]_R)\ott t^{d([\tau]_R)}] & \eqref{exteqn}\\
&=[\b[\ga]_R,\b[\tau]_R] & \eqref{goldmanbracket}
\end{align*}

  \end{proof}


The lemma below is included to justify a step in the proof that $\b$ is a homomorphism.

  \begin{lemma}\label{dhom}
Suppose $[\ga]_R,[\eta]_R \in \hat{\pi}_R(S)$, $\ga \in [\ga]_R$, $\eta\in [\eta]_R$ and $p\in \ga\cap\eta$. Then
    \begin{enumerate}
    \item If $\ga\cdot_p \eta$ is not homotopic to the contractible curve and $p$ is not the corner of a bigon then $d$ is a homomorphism,
      $d([\ga \cdot_p \eta]_R) = d([\ga]_R) + d([\eta]_R)$.
    \item If $\ga$, $\eta$ or, for any $p\in \ga\cap \eta$, $\ga\cdot_p \eta$ is homotopic to the contractible curve then $[[\ga]_R,[\eta]_R] = 0$.
       \end{enumerate}
  \end{lemma}
  \begin{proof}
{\it 1.} By equivariance, we may assume,  without loss of generality, that $\ga$ is unobstructed and so $d([\ga]_R) = 0$. Set $n=-d([\eta]_R)$ so that $d(t^n[\eta]_R) = 0$ and $t^n\eta$ is unobstructed. Now Lem. \ref{prodlem}  shows $\ga \cdot t^n\eta$ is unobstructed or $\ga \cdot_p t^n\eta$ is nulhomotopic. So by assumption, $\ga \cdot t^n\eta$ must be unobstructed and therefore $d([\ga \cdot t^n\eta]_R) = 0$. Since $d$ is $\ZZ$-equivariant, $d(t^n [\eta]_R) = d([\ga \cdot t^n\eta]_R)$ implies $d([\ga \cdot_p \eta]_R) = d([\ga]_R) + d([\eta]_R)$.

{\it 2.} If $\ga$ or $\eta$ are contractible then we can make them small and distant, this shows that the geometric intersection number is zero.
If $\ga \cdot_p \eta$ is contractible then shrinking $\ga$ and $\eta$ along concentric circles $S_r = \{ z : \vnp{z} = r \}$ in the disk $D^2 = \{ z \in \CC : \vnp{z} \leq 1\}$ defining the nulhomotopy shows that their geometric intersection number is zero. For any other basepoint $q$, $\pi_R(S,q) = y^{-1} \cdot \pi_R(S,p) \cdot y$, for a path $y : q\to p$, and conjugation takes identity to identity.
\end{proof}

\begin{rmk}\label{cyclicgrrmk}
  If the function $d$ was always a homomorphism then it would induce a
  homomorphism $\pi_R(S,p)\to \ZZ$ which, by virtue of factoring through the
  abelianization $\pi_R(S,p)^{ab}$, must factor through a homomorphism
  $\ZZ/\chi\to \ZZ$.

Here is a more concrete example.  Suppose that the genus of $S$ is $2$,
so $\chi = -2$.  Set $x=[a_1,b_1]$ and $y=[a_2,b_2]$ in
Eqn. \eqref{pipres}. So that $xy = r$. Now $d(x) = 0$ and $d(y) = 0$ because
they are unobstructed curves.  But $0=d(xy) = d(r) = d(f^2) = 2$ because of
the relation $r = f^{2}$ from Eqn. \eqref{spres}.
\end{rmk}

\begin{question}
  Theorem \ref{reggoldmanthm} suggests that $\grs$ is an analogue of
  the loop algebra $\cL\gs$ of the Goldman Lie algebra.  Is there a central
  extension of $\hat{\g}_R(S)$ corresponding to an affinization of the
  Goldman Lie algebra?
  \end{question}

\section{Induced grading on HOMFLY-PT algebras}\label{homflyptsec}

In this section we use the grading on the regular Goldman Lie algebra to
construct a grading on the HOMFLY-PT skein algebra. We first recall some
definitions.

\begin{defn}[Conway triple]
Let $L_+$, $L_-$ and $L_0$ be link diagrams on a surface, the triple $(L_+,L_-,L_0)$ is called a {\em Conway triple} if $L_+$, $L_-$, and $L_0$ are identical outside of some small neighborhood $U$, and within $U$ appear as in Figure \ref{fig:triple} below.
\end{defn}

\begin{figure}[h]  
\centering
\begin{center}
\begin{overpic}[scale=0.7]
{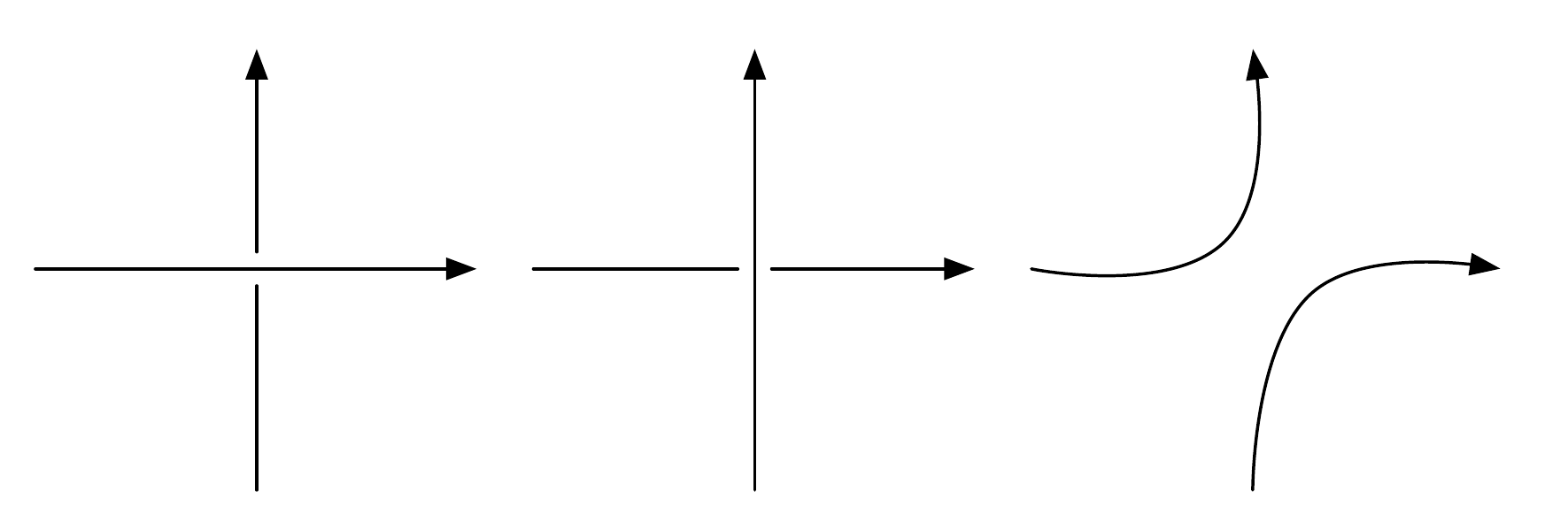}
\put(54,-5){$L_+$}
\put(168,-5){$L_-$}
\put(281,-5){$L_0$}
\end{overpic}
\end{center}
  \caption{Diagram showing the difference between $L_+,L_-,L_0$ within $U$}
\label{fig:triple}
\end{figure}

A skein algebra is usually defined as a quotient of either framed isotopy
classes of framed oriented links in a thickening of the surface $S\times\I$
or oriented link diagrams considered up to Reidemeister moves II and III by
a relation among the Conway triples introduced above.  The equivalence
between these two approaches is discussed by L. Kauffman in \cite{Kauffman}.

We'll stick with link diagrams.

\begin{defn}[HOMFLY-PT skein algebra]\label{skalgdef}
Let $\aL(S)$ denote the set of oriented link diagrams on a surface $S$. Let $R=\mathbb{Z}[x^{\pm 1},q^{\pm 1},v^{\pm 1}]$. The {\em HOMFLY-PT skein algebra} $\hqs$ is the quotient of the $R$-span of $\aL(S)$ by the three types of relations below.
\begin{enumerate}
\item When $(L_+,L_-,L_0)$ is a Conway triple,
$$x \CPPic{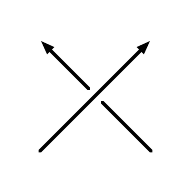} -\quad x^{-1} \CPPic{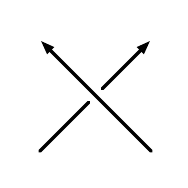} = (q-q^{-1})\CPPic{orres}$$

\item When $L_2$ is obtained from $L_1$ via a Reidemeister I untwisting move 
$$\CPPic{borres2} = \quad v \CPPic{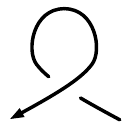} \conj{ and }\quad \CPPic{borres1} = \quad v^{-1} \CPPic{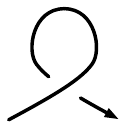}$$

\item If $L_2$ is obtained from $L_1$ via a Reidemeister II or III move then they are required to be equal:  $L_1=L_2$.
\end{enumerate}
\end{defn}

For more details, see \cite{morton}.

\newcommand{\oh}{\omega_H}
\newcommand{\p}{\rho}
\newcommand{\proj}{\p_*}
\newcommand{\B}{B}

Now the natural projection $\p : S\times \I \to S$ takes framed oriented knots in $S\times \I$ to immersed curves in $S$. A regular framed isotopy between two knots is mapped by $\p$ to a regular homotopy between their projections. By Kauffman's work, projection defines a map 
$\p_* : \aK(S) \to \hat{\pi}_R(S)$
from oriented knot diagrams $\aK(S)$ on $S$ to $\hat{\pi}_R(S)$ which preserves Reidemeister moves II and III. An oriented crossing projects to a self-intersection
$$\CPPic{pcross} \rightsquigarrow \CPPic{spcross}.$$
This map extends to a map on links
$$\p_* : \aL(S) \to \ZZ\inp{\hat{\pi}_R(S)}$$
by setting $\p_*(L) := \sum_i \p_*(L_i)$ when $L=\sqcup_i L_i$ for $L_i\in \aK(S)$.

By construction $\hqs$ is $\ZZ$-spanned by elements of the form 
$$\B(S) := \{ x^n v^m q^l L  : n,m,l\in \ZZ \textnormal{ and } L\in \aL(S)\}.$$
The winding number is defined in terms of this spanning set
\begin{equation}\label{skgreq}
  \omega_H(x^n v^m q^l L) := m+\omega_R(\p_*(L))
  \end{equation}
where $\omega_R :\ZZ\inp{\hat{\pi}_R(S)}\to\ZZ/\chi$ is the linear extension of Eqn. \eqref{windingnumbereq}. 
The proposition below shows that this definition respects the relations in Def. \ref{skalgdef}.

\begin{prop}\label{relprop}
The map $\omega_H : \B(S) \to \ZZ/\chi$ from Eqn. \eqref{skgreq} extends to the skein algebra $\hqs$ along the map $\B(S) \to\hqs$.

\end{prop}

\begin{proof}
In Def. \ref{skalgdef}, $\hqs$ is defined to be a quotient of $\ZZ\inp{\B(S)}$ by three types of relations. In order to show that $\omega_H$ extends to the quotient, we check that it respects each of these three types of relations.

In the first case the relation involves a Conway triple
$(L_+,L_-,L_0)$. Since $L_+$ and $L_-$ only differ at a crossing, their
projections $\p_*(L_+) = \p_*(L_-)$ are identical and this implies 
$$\omega_H(xL_+) = \omega_H(x^{-1}L_-).$$
The geometric argument for Prop. \ref{wadditiveprop} shows that $\omega_R(\p_*(L_+)) = \omega_R(\p_*(L_0))$, so
$$\omega_H(qL_0)=\omega_H(q^{-1}L_0)=\omega_H(xL_+).$$
Thus all of terms in the skein relation are in the same graded degree.

In the second case, suppose that $L_2$ is obtained from $L_1$ via a
Reidemeister I twist, so that $L_1 = vL_2$ in $\hqs$. By Corollary
\ref{fishtailpm1}, it follows the winding number of their projections differ
by 1,
$$\omega_R(\p_*(L_1)) - 1 = \omega_R(\p_*(L_2))$$
So $\omega_H(L_1)=\omega_H(vL_2)$.

Lastly, suppose that $L_2$ is obtained from $L_1$ by a Reidemeister II or
III move. Then by Kauffman's work, $L_1$ and $L_2$ are regularly isotopic
and so their projections are regularly freely homotopic. Thus in $\g_R(S)$,
$\proj(L_1)=\proj(L_2)$.
It follows that  $\omega_H(L_1)=\omega_H(L_2)$.
\end{proof}

The theorem below uses the winding map determined by the proposition above to define a grading on the skein algebra.

\begin{theorem}
  The HOMFLY-PT skein algebra $\hqs$ is graded as a ring by the winding number. In more detail, there is a decomposition
  $$\hqs = \bigoplus_{a\in\ZZ/\chi} \hqs_a$$
which respects the product: if $l\in \hqs_a$ and $l'\in \hqs_b$ then $l\cdot l ' \in \hqs_{a+b}$ via the map
$$\hqs_a \ott_{\ZZ[x^{\pm 1}, q^{\pm 1}]} \hqs_b \xto{\cdot} \hqs_{a+b}$$
  \end{theorem}
\begin{proof}
  By Prop. \ref{relprop}, the relations defining the HOMFLY-PT skein algebra respect the winding number.
So for each $a\in\ZZ/\chi$, the $a$th homogeneous component $\hqs_a$ of $\hqs$ is given by
  $$\hqs_a := \ZZ\inp{x\in\B(S) : \omega_H(x) = a}.$$

We conclude by checking that the grading is additive with respect to multiplication. If $l\in \hqs_a$ and $l'\in \hqs_b$ then $l$ is a $\ZZ$-linear combination of elements in $\B(S)$ with winding number $a$, same for $l'$, so it suffices to check additivity for monomials $l = x^a q^b v^c L$ and $l' = x^\ell q^m v^n L'$ in $\B(S)$. Then
\begin{align*}
  \omega_H(l\cdot l') &=  \omega_H (x^{a+\ell} q^{b+m} v^{c+n}L \cdot L')\\
  &= c+n + \omega_R(L\cdot L') \\
  &= c+n + \omega_R(L) + \omega_R(L')\\
  &= \omega_H(l) + \omega_H(l')\\
  \end{align*}
where the equality $\omega_R(L\cdot L') = \omega_R(L) + \omega_R(L')$ follows from the observation that the individual knot components in the stacking product are the same as those of the disjoint union of $L$ and $L'$,  $L\cdot L' = (\sqcup_i L_i) \sqcup (\sqcup_j L_j)$.

\begin{rmk}
  Since a grading of the HOMFLY-PT skein algebra attached to the framing
  must respect the oriented loop product appearing in 
the relation 1 of Def. \ref{skalgdef}, it isn't a $\ZZ$-grading for the same reasons as
  those mentioned in Rmk. \ref{cyclicgrrmk}.
\end{rmk}

\end{proof}

\bibliography{newvs}{}
\bibliographystyle{amsalpha}
\Addresses

\end{document}